\documentclass[11pt]{amsart}
\usepackage{amsmath,amsmath,latexsym,amssymb,amsfonts,amsbsy, amsthm}
\usepackage{mathrsfs}
\usepackage{fancyhdr}
\usepackage{latexsym}
\usepackage{multicol,graphics}
\usepackage{amssymb}
\usepackage{graphicx}
\usepackage{color}
\usepackage{bm}
\allowdisplaybreaks

\newcommand{\N}{{\mathbb N}}
\newcommand{\R}{{\mathbb R}}
\def\K{\mathcal{K}}

\newtheorem{thm}{Theorem}[section]
\newtheorem{lem}[thm]{Lemma}

\newtheorem{cor}[thm]{Corollary}
\newtheorem{pro}[thm]{Proposition}
\newtheorem{rem}[thm]{Remark}

\numberwithin{equation}{section}

\begin{document}

\title[]{Exact controllability and stability of the Sixth Order Boussinesq equation}
\author{Shenghao Li,  Min Chen, Bing-Yu Zhang}
\date{}
\begin{abstract}
The article studies the exact controllability and the stability of the sixth order Boussinesq equation
 \[    u_{tt}-u_{xx}+\beta u_{xxxx}-u_{xxxxxx}+(u^2)_{xx}=f, \quad \beta=\pm1, \]
 on the interval $S:=[0,2\pi]$ with periodic boundary conditions.

 It is shown that the system is locally exactly controllable in the classic Sobolev space, $H^{s+3}(S)\times H^s(S)$ for $s\geq 0$, for ``small'' initial and terminal states. It is also shown that if $f$ is assigned as an internal linear feedback, the solution of the system is uniformly exponential decay to a constant state in $H^{s+3}(S)\times H^s(S)$ for $s\geq 0$ with ``small" initial data assumption.

\end{abstract}



\maketitle


\section{Introduction}
The Boussinesq equation,
\begin{equation}\label{bad}
u_{tt}-u_{xx}+(u^2)_{xx}- u_{xxxx}=0,
\end{equation}
was originally derived by J. Boussinesq (1871) \cite{25} in his study on propagation of small amplitude, long waves on the surface of water. It possesses some special traveling wave solutions called  solitary waves, and it is the first equation that gives a mathematical explanation to the phenomenon of solitary waves which was discovered and reported by Scott Russell in 1830s. The original Boussinesq equation has been used in a considerable range of applications such as coast and harbor engineering, simulation of tides and tsunamis.

However, the original Boussinesq equation (\ref{bad}) has a  drawback, it is  ill-posed for its initial-value problem in the sense that  a slight difference in initial data might evolve into a large change in solutions. This can be seen, for example,
by considering its linear equation as
\begin{equation*}
    u_{tt}-u_{xxxx}:= (\partial_t+\partial_{xx})(\partial_t-\partial_{xx})u=0.
\end{equation*}
The ``$\partial_t-\partial_{xx}$'' can be treated as the heat equation and it is well-posed, but ``$\partial_t+\partial_{xx}$'', the backward heat equation, is ill-posed. One way to correct this ill-posedness issue is to alter the sign of the fourth order derivative term, which leads to  the ``good" Boussinesq equation
 \begin{equation}\label{good}
 u_{tt}-u_{xx}+(u^2)_{xx}+ u_{xxxx}=0.
 \end{equation}
Similarly, its well-posedness can be seen by considering its linearized equation as
\begin{equation*}
    u_{tt}+u_{xxxx}:= (\partial_t+i\partial_{xx})(\partial_t-i\partial_{xx})u=0,
\end{equation*}
a combination of the Schr\"odinger and the reversed Schr\"odinger equations whose initial-value problems are both well-posed. But, due to  the change of the sign to the Boussinesq equation, the ``good" Boussinesq equation cannot be well  justified as a  physical modeling of water waves as the original Boussinesq equation.
To remedy the case,  Christov, Maugin and Velarde \cite{28} modified the original Boussinesq's physical modeling and derived the sixth order Boussinesq equation,
\[u_{tt}-u_{xx}+\beta u_{xxxx}-u_{xxxxxx}+(u^2)_{xx}=0, \quad \beta=\pm1.\]
Its well-posedness can be seen from the linearized  equation:
\begin{equation*}
    u_{tt}-u_{xxxxxx}:=(\partial_t+\partial_{xxx})(\partial_t-\partial_{xxx})u=0,
\end{equation*}
which can be written as a coupled linear KdV equations, where  the KdV and the reversed KdV equations are both  well-posed for their initial value problems.

The sixth order Boussinesq equation was also proposed in modeling the nonlinear lattice dynamics in elastic crystals by Maugin \cite{66}. In addition,
Feng  et al \cite{43} studied the solitary waves as well as their interactions of the sixth order Boussiensq equation. Kamenov \cite{48} obtained an exact periodic solution through the Hirota's bilinear transform method. Moreover, the initial value problem of the sixth order Boussinesq equation and  its initial boundary value problem have been studied in \cite{34,35,125,96,98,79}. However, unlike the well-posedness issues, the control problems of the sixth order Boussinesq equation have not yet been studied.

Since the 1980s, the   control theory of  the nonlinear dispersive wave equations have attracted a lot of attentions due to the development of the mathematical theory on these equations. In particular, the theories on control of the KdV equation were intensively advanced through many people's work \cite{126,127,128,129,131,130,72,73,132}. Other equations such as Kawahara, Boussinesq and nonlinear Schr\"odinger equations were also studied \cite{130,137,134,122,133}. Our goal is to study the control problems of the sixth order Boussinesq equation based on the ideas on the KdV and Boussinesq equations (c.f. \cite{72,73,122}), due to its similarities to those equations (c.f. \cite{125,96,98}).

Our main concern of this article is the  distributed control problem of the sixth order Boussinesq equation with periodic boundary conditions,
      \begin{equation}\label{f2}
        \begin{cases}
        u_{tt}-u_{xx}+\beta u_{xxxx}-u_{xxxxxx}+(u^2)_{xx}=f, \quad x\in S,\\
        \partial_x^k u(2\pi,t)=\partial_x^k u(0,t), \quad k=0,1,...,5.
        \end{cases}
              \end{equation}
The plan is to address the following control problems:
\begin{itemize}
  \item Let $T>0$ be given. Provide the initial condition, $(\varphi_0,\psi_0)$, and the terminal condition, $(\varphi_T,\psi_T)$, in an appropriate space, can one find a control $f$ such that the system \eqref{f2}
        admits a solution $u=u(x,t)$ satisfying
        \[u(x,0)=\varphi_0, u_t(x,0)=\psi_0,\]
        \[u(x,T)=\varphi_T, u_t(x,T)=\psi_T?\]
        \item Can one find a linear feedback control law
        \[f={\mathcal H} u \]
        such that the resulting closed-loop system is exponentially stable?
\end{itemize}

To start our study, certain restrictions on the system are required. We note that, for a smooth solution $u(x,t)$ of the unforced equation (i.e. $f(x,t)=0$ in \eqref{f2}), it can be checked that
\[\frac{d}{dt}\int_S u_{t}(x,t)dx=0,\]
for any $t\in \R$. Therefore,
\[\int_S u_t(x,t)dx=\int_Su_t(x,0)dx\]
and
\begin{equation}\label{conserve}
\int_S u(x,t)dx=\int_S u(x,0)dx+t\int_S u_t(x,0)dx
\end{equation}
for any $t\in \R$. From the original derivation of the equation (see \cite{28}), it describes the behavior of water in a shallow channel. It is then natural to think \eqref{conserve} as the ``conserved volume'' or mass.  Hence, in order to have  $\int_S u(x,t)dx$ conserved, one can choose $u_t(x,0)$ such that $\int_S u_t(x,0)dx=0$. Moreover, with such choice, it follows that the quantity $\int_S u_t(x,t)dx$ is also conserved. Therefore, in order to keep the quantities still conserved on the forced system \eqref{f2} with the control $f(x,t)$, we require that
\begin{equation}\label{f3}
    \int_S f(x,t)dx=0, \quad \forall t\in \R.
\end{equation}

We next impose a priori restriction on the control $f(x,t)$. Let us suppose that $g(x)$ is a smooth function defined for $x\in S$ such that
\[[g]:=\frac{1}{2\pi}\int_S g(x)dx=1,\]
where $[g]$ denotes the mean value of the function $g$ over $S$. We set the control function in the form
\begin{equation}\label{f4}
    f(x,t)=Gh:=g(x)\left(h(x,t)-\int_S g(y)h(y,t)dy\right).
\end{equation}
Then, $h(x,t)$ can be considered as a new control input and it can  be checked the function $f$ defined in \eqref{f4} satisfies the restriction \eqref{f3}. Now, we have the control system written in the form
\begin{equation}\label{con0}
     u_{tt}-u_{xx}+\beta u_{xxxx}-u_{xxxxxx}+(u^2)_{xx}=G h, \quad x\in S,
\end{equation}
with $[u_t(x,0)]=0$.

Before we state the main results of this paper, some preliminaries and notations are needed. Let $H^s(S)$ for $s\geq 0$ be the space of all functions in the form
\[v(x)=\sum^{\infty}_{-\infty} v_k e^{ikx}\]
such that
\begin{equation}\label{sls}
\left\{\sum^{\infty}_{-\infty} |v_k|^2(1+|k|)^{2s} \right\}^{\frac12}<+\infty.
\end{equation}
The left side of \eqref{sls} is a Hilbert norm for $H^s(S)$ and we  denote it as $\|v\|_s$.  In addition, for any $s\geq 0$, we set
    \[X^s:=H^{s+3}(S)\times H^s(S)\]
    and define its norm
    \[\|\vec{u}\|_{X^s}:=\left(\|u\|_{s+3}^2+\|v\|_s^2\right)^{\frac12}\]
    for any $\vec{u}=\begin{bmatrix}
   u \\
    v\end{bmatrix}\in X^s$.

The following theorem is the main result on the exact controllability problem of \eqref{con0}.

\begin{thm}\label{c2}
 Let $T>0$ and $s\geq 0$ be given. There exists a $\delta>0$ such that for any  $(\varphi_0,\psi_0)\times(\varphi_T,\psi_T)\in X^s\times X^s$  with $[\varphi_0]=[\varphi_T]$ and $[\psi_0]=[\psi_T]=0$ satisfying
 \begin{equation}\label{g01}
 \|\varphi_0\|_{s+3}+\|\psi_0\|_{s}\leq \delta, \quad \|\varphi_T\|_{s+3}+\|\psi_T\|_{s}\leq \delta,
 \end{equation}
 then one can find a control function $h\in L^2(0,T;H^s(S))$  such that the equation
\[u_{tt}-u_{xx}+\beta u_{xxxx}-u_{xxxxxx}+ (u^2)_{xx}= Gh,\quad x\in S,\]
has a solution $u\in C(0,T;H^{s+3}(S))\times C^1(0,T;H^s(S))$ with
\[u(x,0)=\varphi_0(x), \quad u_t(x,0)=\psi_0(x),\]
\[u(x,T)=\varphi_T(x), \quad u_t(x,T)=\psi_T(x).\]
 \end{thm}

The other main result in this paper is the stability of the system \eqref{con0}. We will show that by given the following feed back law
\[h=-\K u_t(x,t),\]
for some $\K>0$, the controlled solution $u(x,t)$ of the following closed-loop system
\begin{equation}\label{closel0}
    \begin{cases}
    u_{tt}-u_{xx}+\beta u_{xxxx}-u_{xxxxxxx}+(u^2)_{xx}=-\K Gu_t,  \ \ \ x\in S,\\
    u(x,)=\varphi_0(x), u_t(x,0)=\psi_0(x),\\
    \end{cases}
\end{equation}
should tend to the constant state $\tilde{u}(x):=[\varphi_0]$ as $t\rightarrow \infty$ if $[\psi_0]=0$.
\begin{thm}\label{stable1}
For $s\geq 0$, there exists some positive constants $\delta$, $M$, and $\sigma$ such that every solution of the system \begin{equation}\label{stables}
    \begin{cases}
    u_{tt}-u_{xx}+\beta u_{xxxx}-u_{xxxxxxx}+(u^2)_{xx}=-\K Gu_t,  \ \ \ x\in S,\\
    u(x,)=\varphi_0(x), u_t(x,0)=\psi_0(x),\\
    \end{cases}
\end{equation}
where $(\varphi_0,\psi_0) \in X^s$ with  $[\psi_0]=0$ and $\|(\varphi_0,\psi_0)\|_{X^s}\leq \delta$ satisfies
\begin{equation}\label{nlest}
     \|(u(\cdot,t), u_t(\cdot,t))-([\varphi_0],0)\|_{X^s}\leq M e^{-\sigma t}\|(\varphi_0,\psi_0)-([\varphi_0] ,0)\|_{X^s}.
\end{equation}
\end{thm}

The article is organized as follows: Section 2 is devoted to consider the controllability problems. In subsection 2.1, we first address the well-posedness issue for forced system. In subsection 2.2, we conduct a spectral analysis of the operator
\[A=\begin{bmatrix}
   0&1 \\
    \partial_x^2+\beta\partial_x^4+\partial_x^6 &0\end{bmatrix}\]
    defined in the space $H^{s+6}(S)\times H^s(S)$ for $s\geq 0$. We show it is a discrete spectral operator and its eigenvectors form a Riesz basis of the space $X^s$. This allows us to proceed the techniques adapted in \cite{72,73,122} for KdV and Boussinesq equations. The proof of our main results Theorem \ref{c2} will be given in subsection 2.3. Section 3 considers the stability problem. In subsection 3.1, we show the exponential decay for the solution of the linear problem. In subsection 3.2, we move on to the nonlinear problem and prove the Theorem \ref{stable1}.
\section{Exact Controllability problem}

\subsection{well-posedness}
 We first establish the well-posedness of the initial value problem (IVP) of the forced sixth order Boussinesq equation on a periodic domain $S$,
\begin{equation}\label{f01}
    \begin{cases}
    &u_{tt}-u_{xx}+\beta u_{xxxx}-u_{xxxxxx}+(u^2)_{xx}=f, \quad x\in S, \\
    &u(x,0)=\varphi_0(x), u_{x}(x,0)=\psi_0(x).
    \end{cases}
\end{equation}

We rewrite the IVP \eqref{f01} into the following first order evolution problem,
\begin{equation}\label{w2}
    \frac{d}{dt}\vec{u}=A \vec{u}+F(\vec{u})+\vec{g}, \quad \vec{u}(0)=\vec{u}_0,
\end{equation}
where
\[\vec{u}=\begin{bmatrix}
   u \\
    u_t\end{bmatrix},\quad A=\begin{bmatrix}
   0 & 1 \\
    \partial_{x}^2-\beta \partial_{x}^4+\partial_{x}^6& 0\end{bmatrix},\quad F(\vec{u})=\begin{bmatrix}
   0 \\
    -(u^2)_{xx},\end{bmatrix}\]
    and
    \[\vec{g}=\begin{bmatrix}
   0 \\
    f\end{bmatrix}, \quad \vec{u}_0=\begin{bmatrix}
   \varphi_0 \\
    \psi_0\end{bmatrix}.\]
    It can be checked that the operator $A$ is a linear operator from $X^s$ to $X^s$ with $D(A)=H^{s+6}(S)\times H^s(S)$ and it generates an isomorphic group $W(t)$ on the space $X^s$ for any $s\geq 0$. The following  proposition comes from  the standard semigroup theory (c.f. \cite{123}).

    \begin{pro}\label{p1}
    Let $s\geq0$ and $T>0$ be given. There exists a constant $C>0$ such that
    \begin{equation}\label{w3}
        \sup_{t\in[0,T]}\|W(t)\vec{u}\|_{X^s}\leq C \|\vec{u}\|_{X^s}
    \end{equation}
    for any $\vec{u}\in X^s$ and
    \begin{equation}\label{w4}
        \sup_{t\in[0,T]} \left\|\int^t_0 W(t-\tau)\vec{f}d\tau\right\|_{X^s}\leq C\|\vec{f}\|_{L^1(0,T;X^s)}
    \end{equation}
    for any $\vec{f}\in L^1(0,T;X^s)$.
    \end{pro}
    We can now state the theorem of the well-posedness for IVP \eqref{w2}.
    \begin{thm}\label{f02}
Let $s\geq 0$ and $T>0$ be given. For any $\vec{u}_0\in X^s$ and $\vec{g}\in L^1(0,T;X^s)$, there exists a $T^*>0$, depending only on $\|\vec{u}_0\|_{X^s}$ and $\|\vec{g}\|_{L^1(0,T;X^s)}$, such that the IVP \eqref{w2} has a unique solution $\vec{u}\in C(0,T^*;X^s)$.
\end{thm}

    The proof of Theorem \ref{f02} is established on a standard contraction mapping process based on Proposition \ref{p1}, which is very much similar to the work in \cite{73}, therefore omitted.

\subsection{Spectral analysis}
In this subsection, we analysis the spectral of the operator $A$ defined in \eqref{w2}. Recall that
\[
A=\begin{bmatrix}
   0&1 \\
    \partial_x^2-\beta\partial_x^4+\partial_x^6 &0\end{bmatrix}.\] We define
\begin{equation}\label{sp1}
E_{1,0}=\begin{bmatrix}
   1 \\
    0\end{bmatrix},\quad
 E_{2,0}=\begin{bmatrix}
   0 \\
    1\end{bmatrix},
\end{equation}
and
\begin{equation}\label{sp2}
E_{1,k}=\frac{1} {k^{3}}\begin{bmatrix}
   e^{ikx} \\
    0\end{bmatrix},\quad
 E_{2,k}=\begin{bmatrix}
   0 \\
    e^{ikx}\end{bmatrix},
\end{equation}
for $k=\pm 1, \pm 2,...$. Direct computation leads to the following conclusions.
 \begin{itemize}
   \item For $k=\pm1, \pm2,...$, we notice that
\[A(E_{1,k},E_{2,k})=(E_{1,k},E_{2,k})M_k\]
with
\[M_k=\begin{bmatrix}
   0 & k^3 \\
    -\frac1k(k^4+\beta k^2+1) & 0\end{bmatrix}.\]
    The matrix $M_k$ has eigenvalues
\[\lambda_{1,k}=i\sqrt{k^2(k^4+\beta k^2+1)},\quad \lambda_{2,k}=-i\sqrt{k^2(k^4+\beta k^2+1)},\]
with corresponding eigenvectors
\begin{equation}\label{sp3}
\vec{e}_{1,k}=\begin{bmatrix}
   1\\
    \frac{\lambda_{1,k}}{k^3}\end{bmatrix},\quad \vec{e}_{2,k}=\begin{bmatrix}
   1\\
    \frac{\lambda_{2,k}}{k^3}\end{bmatrix}.
\end{equation}
    \item For  $k=0$, we notice that
\[A(E_{1,k},E_{2,k})=(E_{1,k},E_{2,k})M_k,\quad \mbox{with}\quad M_k=\begin{bmatrix}
   0 & 1 \\
    0 & 0\end{bmatrix}.\]
The  matrix $M_k$ has eigenvalues  $\lambda_{1,k}=\lambda_{2,k}=0$ with corresponding eigenvectors
\begin{equation}\label{sp4}
\vec{e}_{1,k}=\vec{e}_{2,k}=\begin{bmatrix}
   1 \\
    0\end{bmatrix}
.
\end{equation}
 \end{itemize}
Therefore, these yield:
\begin{itemize}
  \item For $k=\pm1, \pm2,...$,
\begin{align*}
A(E_{1,k},E_{2,k})(&\vec{e}_{1,k},\vec{e}_{2,k})=(E_{1,k},E_{2,k})M_k (\vec{e}_{1,k},\vec{e}_{2,k})\\
&=(\lambda_{1,k}(E_{1,k},E_{2,k})\vec{e}_{1,k},\lambda_{2,k}(E_{1,k},E_{2,k})\vec{e}_{2,k}).
\end{align*}
One can deduce that the operator $A$ has eigenvalues $\lambda_{1,k}$ and $\lambda_{2,k}$ with corresponding eigenvectors
\begin{equation}\label{v}
    \vec{\eta}_{1,k}=(E_{1,k},E_{2,k})\vec{e}_{1,k},\quad \vec{\eta}_{2,k}=(E_{1,k},E_{2,k})\vec{e}_{2,k}.
\end{equation}
  \item For $k=0$, we can also show  that the operator $A$  has an eigenvalue $\lambda_k=0$ with the corresponding eigenvector
\[\vec{\eta}_k=\begin{bmatrix}
   1 \\
    0\end{bmatrix}.\]
\end{itemize}
    In addition,
    \[L_k=(\vec{e}_{1,k},\vec{e}_{2,k})\rightarrow
\begin{bmatrix}
   1&1 \\
    i&-i\end{bmatrix}
\]
as $k\rightarrow +\infty$, which follows
\[\lim_{k\rightarrow \infty }\mbox{det} L_k =-2i\neq 0,\]
thus $\{\vec{e}_{1,k},\vec{e}_{2,k}\}$, for $k=\pm 1,\pm 2,...$, are linearly independent.
        Moreover, since $\{E_{1,k},E_{2,k}\}$, $k=0, \pm 1,\pm 2,...$ form an orthogonal basis for the space $X^s$, according to \cite{119}, we have that $\{\vec{\eta}_0, \vec{\eta}_{1,k},\vec{\eta}_{2,k},  k=\pm1,\pm2,...\}$ forms a Riesz basis for the space $X^s$. 

        We denote
        \[m_{j,k}:=\|\vec{\eta}_{j,k}\|_{X^s},\]
        and
        \begin{equation}\label{sp5}
        \hat{\phi}_{j,k}:=\vec{\eta}_{j,k}/m_{j,k},
    \end{equation}
        for $j=1,2$ and $k=\pm1, \pm2,...$. It can be verified that,
        \[\{\hat{\phi}_0=\vec{\eta}_0, \quad \hat{\phi}_{j,k}, \quad \mbox{for } j=1,2 \mbox{ and } k=\pm1,\pm2,...\}\]
        forms an orthonormal basis for the space $X^s$, that is,
        \[\langle\hat{\phi}_{j,k},\hat{\phi}_{l,m}\rangle=\begin{cases}
        &1, \quad \mbox{if }j=l, k=m,\\
        &0, \quad \mbox{otherwise}.
        \end{cases}\]
  \begin{rem}
  To obtain the above orthogonality between vectors, we need to adapt an equivalence definition of the norm in $H^s(S)$, that is, for
  \[v(x)=\sum^\infty_{k=-\infty} v_k e^{ikx},\]
  we have the following equivalence relation,
  \[\|v\|_{s}\approx \sum^{\infty}_{-\infty}(k^2(1+\beta k^2+k^4))^{s/3}|v_k|^2.\]
  \end{rem}
  Now, we state the following theorem based to above analysis.
  \begin{thm}\label{sa}
  Define
  \[Q:=L^2(S)\times L^2(S).\]
  Let
  \begin{equation}\label{eva}
  \lambda_n=\begin{cases}
  i\sqrt{n^2(n^4+\beta n^2+1)}, \quad &n=1,2,...,\\
  -i\sqrt{n^2(n^4+\beta n^2+1)}, \quad &n=-1,-2,...,
    \end{cases}
  \end{equation}
  \begin{equation}\label{eve1}
  \phi_{1,n}=\begin{cases}
  \hat{\phi}_{1,n}, \quad &n=1,2,...,\\
  \hat{\phi}_{2,n}, \quad &n=-1,-2,...,
    \end{cases}
  \end{equation}
and
  \begin{equation}\label{eve2}
  \phi_{2,n}=\begin{cases}
  \hat{\phi}_{1,-n}, \quad &n=1,2,...,\\
  \hat{\phi}_{2,-n}, \quad &n=-1,-2,....
    \end{cases}
  \end{equation}
  Then
  \begin{description}
    \item[(a)] The spectrum of the operator $A$ consists of eigenvalues $\{\lambda_n\}^\infty_{n=-\infty}$ in which $\lambda_0=0$ has the corresponding single eigenvector $\phi_0= (1,0)^T$ and each $\lambda_n$, $n=\pm1, \pm2,...$, has the corresponding double eigenvectors $\phi_{j,n}$, $j=1,2$.
    \item[(b)] $\{\phi_0, \phi_{j,n}, j=1,2, n=\pm1,\pm2,...\}$ forms an orthonormal basis for the space $X^s$ and any $\vec{w}\in X^s$  has the following Fourier series expansion
        \begin{equation}\label{exp}
        \vec{w}=\alpha_0 \phi_0+\sum^\infty_{n=-\infty}(\alpha_{1,n}\phi_{1,n}+\alpha_{2,n}\phi_{2,n})
        \end{equation}
        with
        \[\alpha_0=\langle\vec{w},\phi_0\rangle_{Q}, \quad  \alpha_{1,n}=\langle\vec{w},  \phi_{1,n}\rangle_{Q}, \quad \alpha_{2,n}=\langle\vec{w}, \phi_{2,n}\rangle_{Q},\]
        for $n=\pm1, \pm2,...$.
  \end{description}
  \end{thm}

  \subsection{Linear problem}
  We  start to address the exact controllability of the linear problem
  \begin{equation}\label{G}
    \begin{cases}
    u_{tt}-u_{xx}+\beta u_{xxxx}-u_{xxxxxx}=Gh, \quad x\in S,\\
    u(x,0)=\varphi_0(x),    u_t(x,0)=\psi_0(x),
    \end{cases}
  \end{equation}
  with $[\psi_0]=0$. Re-write the system \eqref{G} as a first order evolution problem, that reads,
 \begin{equation}\label{G2}
    \begin{cases}
    \frac{d}{dt}\vec{u}=A\vec{u}+B\vec{h},\\
    \vec{u}(0)=\vec{u}_0,
    \end{cases}
 \end{equation}
 where
 \[\vec{u}=\begin{bmatrix}
   u \\
    u_t\end{bmatrix}, \quad \vec{h}=\begin{bmatrix}
   0 \\
    h\end{bmatrix}, \quad
    B\vec{h}=\begin{bmatrix}
   0 \\
    Gh\end{bmatrix},\quad
    \vec{u}_0=\begin{bmatrix}
   \varphi_0 \\
    \psi_0\end{bmatrix}.
\]
 Now, given $T>0$ and $s\geq 0$, let $\vec{u}_0\in X^s$ and $B\vec{h}\in L^1(0,T;X^s)$, according to Theorem \ref{f02}, the solution  of the system  \eqref{G2}, $\vec{u}(\cdot, t)\in X^s$,  can be written as
\begin{equation}\label{G3}
    \vec{u}(t)=W(t)\vec{u}_0+\int^t_0 W(t-\tau)B\vec{h}(\tau)d\tau.
\end{equation}
Therefore, according to  Theorem \ref{sa}, one can write
\begin{align}\label{G4}
    \vec{u}(t)=\alpha_0 e^{\lambda_0 t} \phi_0 +\sum_{n\neq 0}&(\alpha_{1,n}e^{\lambda_n t}\phi_{1,n}+\alpha_{2,n}e^{\lambda_n t}\phi_{2,n})+\int^t_0  e^{\lambda_0 (t-\tau)} \beta_0(\tau) \phi_0d\tau \nonumber\\
        +&\int^t_0 \sum_{n\neq 0}\left(\beta_{1,n}(\tau)e^{\lambda_n (t-\tau)}\phi_{1,n}+\beta_{2,n}(\tau)e^{\lambda_n (t-\tau)}\phi_{2,n} \right)d\tau
\end{align}
where
\begin{equation}\label{ab0}
\alpha_0=\langle\vec{u}_0, \phi_0\rangle_{Q}=\langle\begin{bmatrix}
   \varphi_0 \\
    \psi_0\end{bmatrix},\begin{bmatrix}
   1\\
    0\end{bmatrix}\rangle_{Q}=[\varphi_0],
\end{equation}
\begin{equation}\label{ab2}
    \beta_0=\langle B\vec{h}, \phi_0\rangle_{Q}=\langle\begin{bmatrix}
   0 \\
    Gh\end{bmatrix},\begin{bmatrix}
   1\\
    0\end{bmatrix}\rangle_{Q}=0,
\end{equation}
and
\begin{equation}\label{ab1}
\alpha_{j,n}=\langle\vec{u}_0, \phi_{j,n}\rangle_{Q}, \quad \beta_{j,n}=\langle B\vec{h}, \phi_{j,n}\rangle_{Q},
\end{equation}
for $j=1,2$ and $n=\pm1, \pm2, ...$ where $\beta_{j,n}$ is unknown and determined by the choice of the control $h$ in $B\vec{h}$. Notice that
\[\langle B\vec{h}, \begin{bmatrix}
   u \\
    v\end{bmatrix}\rangle_{Q}=\langle\begin{bmatrix}
   0 \\
    Gh\end{bmatrix}, \begin{bmatrix}
   u \\
    v\end{bmatrix}\rangle_{Q}=\langle Gh,v\rangle_{L^2(S)}=\langle h,Gv\rangle_{L^2(S)},\]
    since it can be checked that the operator $G$ defined in \eqref{f4} is self-adjoint in $L^2(S)$. We denote $(\phi_{j,n})_2$ to be the second entry of $\phi_{j,n}$, then
  \begin{equation}\label{phi1}
     \beta_{j,n}=\langle B\vec{h},\phi_{j,n}\rangle_{Q} = \langle h,  G((\phi_{j,n})_2) \rangle_{L^2(S)}.
  \end{equation}


Now, we  define
\begin{align*}
{\mathcal{X}}_s=\bigg\{(\vec{u}_0,\vec{u}_T)\in X^s\times X^s, \quad \int_s (\vec{u}_0)_1&dx=\int_s (\vec{u}_T)_1dx,\\
 &\int_s (\vec{u}_0)_2dx=\int_s (\vec{u}_T)_2dx=0\bigg\},
\end{align*}
where $(\vec{u})_j$ denotes the $j-$th entry of $\vec{u}$  with $j=1,2$. The exact controllability of system \eqref{G3} states as follows.
\begin{thm}\label{c0}
Let $T>0$ be given. For any $s\geq0$, there exists a bounded linear operator
\[K_T: \quad {\mathcal X}_s\rightarrow L^2 (0,T; X^s)\]
such that for any $(\vec{u}_0,\vec{u}_T)\in{\mathcal X}_s$, the solution of
\begin{equation}\label{con}
\begin{cases}
    \frac{d}{dt}\vec{u}(t)=A\vec{u}(t)+B K_T(\vec{u}_0,\vec{u}_T),\\
    \vec{u}(0)=\vec{u}_0,
    \end{cases}
\end{equation}
satisfies
\[\vec{u}(T)=\vec{u}_T.\]
In addition, one has
\begin{equation}\label{con1}
\|K_T(\vec{u}_0,\vec{u}_T)\|^2_{L^2(0,T,X^s)}\leq C_T(\|\vec{u}_0\|^2_{X^s}+\|\vec{u}_T\|^2_{X^s}),
\end{equation}
with $C_T$ independent of $\vec{u}_0$ and $\vec{u}_T$.
\end{thm}
\begin{proof}
We first consider the case provided $[\varphi_0]=0$, the general case for $[\varphi_0]=\kappa\neq 0$ will be proved later. Without loss of generality, we can assume that $\vec{u}_T=\vec{0}$, since the system \eqref{con} is time reversible. The exact control problem then becomes to find a $h\in L^2(0,T;H^s(S))$ in equation \eqref{G4} such that $\vec{u}(T)=\vec{0}$, that is,
\begin{align}\label{G5}
    \sum_{n\neq 0}(\alpha_{1,n}&e^{\lambda_n T}\phi_{1,n}+\alpha_{2,n}e^{\lambda_n T}\phi_{2,n})\nonumber\\
        &+\int^T_0 \sum_{n\neq 0}\left(\beta_{1,n}(\tau)e^{\lambda_n (T-\tau)}\phi_{1,n}+\beta_{2,n}(\tau)e^{\lambda_n (T-\tau)}\phi_{2,n} \right)d\tau=\vec{0},
\end{align}
since, by assumption, $\alpha_0=[\varphi_0]=0$ (c.f. \eqref{ab0}). This leads to the solve $\beta_{j,n}$ from the following system,
\[\begin{cases}
\alpha_{1,n} +\int^T_0 \beta_{1,n} e^{-\lambda_n\tau}d\tau=0,\\
\alpha_{2,n} +\int^T_0 \beta_{2,n} e^{-\lambda_n\tau}d\tau=0,
\end{cases}\]
for $j=1,2$ and $n=\pm1,\pm2,...$. Then, combining \eqref{phi1}, we have the necessary conditions for the control $h$ to satisfy \eqref{G5},
\begin{equation}\label{phi2}
    \begin{cases}
\alpha_{1,n} +\int^T_0 e^{-\lambda_n t}\langle h,  G((\phi_{1,n})_2) \rangle_{L^2(S)} dt=0,\\
\alpha_{2,n} +\int^T_0 e^{-\lambda_n t}\langle h,  G((\phi_{2,n})_2) \rangle_{L^2(S)} dt=0.
\end{cases}
\end{equation}

Next, by denoting $p_k=e^{\lambda_k t}$, ${\mathcal P}=:\{p_k, -\infty<k<\infty\}$ forms a Riesz basis for its closed span, $P_T$, in $L^2(0, T)$ (c.f. \cite{124}). Let ${\mathcal L}= \{q_k, -\infty<k<\infty\}$ be the unique dual Riesz basis for $\mathcal P$ in $P_T$, that is, the functions in $\mathcal L$ are the unique elements of $P_T$ such that
\begin{equation}\label{dual}
    \int^T_0 q_k(t)\overline{p_l(t)}dt=\begin{cases}
    0, \quad k\neq l,\\
    1, \quad k=l.
    \end{cases}
\end{equation}
Now, we write the control $h$  presented in \eqref{phi2} in the form,
\begin{equation}\label{h}
    h(x,t)=c_0 q_0 +\sum_{n\neq 0} q_n\Big[c_{1,n}G((\phi_{1,n})_2)+c_{2,n}G((\phi_{2,n})_2)\Big],
\end{equation}
where $c_0$, $c_{1,n}$ and $c_{2,n}$ are to be determined so that the series \eqref{h} satisfies \eqref{phi2} and converges appropriately.

Substituting \eqref{h} into \eqref{phi2}, it yields that $c_{1,n}$, $c_{2,n}$ satisfy
    \begin{align}
    -\alpha_{1,n}=c_{1,n}\langle G((\phi_{1,n})_2), G((\phi_{1,n})_2&) \rangle_{L^2(S)}\nonumber\\
    &+c_{2,n}\langle G((\phi_{2,n})_2), G((\phi_{1,n})_2) \rangle_{L^2(S)},\label{h11}\\
    -\alpha_{2,n}=c_{1,n}\langle G((\phi_{1,n})_2),G((\phi_{2,n})_2&) \rangle_{L^2(S)}\nonumber\\
   & +c_{2,n}\langle G((\phi_{2,n})_2), G((\phi_{2,n})_2) \rangle_{L^2(S)},\label{h12}
    \end{align}
     for $n=\pm1,\pm2,...$. In addition, for simplicity we set $c_0=0$.

 Denote
\begin{align}\label{delta}
\Delta_n=&\begin{vmatrix}
\langle G((\phi_{1,n})_2), G((\phi_{1,n})_2) \rangle_{L^2(S)} & \langle G((\phi_{2,n})_2), G((\phi_{1,n})_2) \rangle_{L^2(S)} \\
\langle G((\phi_{1,n})_2), G((\phi_{2,n})_2) \rangle_{L^2(S)} & \langle G((\phi_{2,n})_2), G((\phi_{2,n})_2) \rangle_{L^2(S)}
\end{vmatrix}\nonumber\\
=& \big\|G((\phi_{1,n})_2)\big\|^2_{0}\big\|G((\phi_{2,n})_2)\big\|^2_{0}-\big|\langle G((\phi_{1,n})_2), G((\phi_{2,n})_2) \rangle_{L^2(S)}\big|^2.\nonumber
\end{align}
Notice that $\Delta_n\neq 0$ for any $n$ since $G((\phi_{1,n})_2)$ and $G((\phi_{2,n})_2)$ are linearly independent. Moreover, as $n\rightarrow \infty$, through direct computation and definition of the operator $G$, we have $|\langle G((\phi_{1,n})_2), G((\phi_{2,n})_2) \rangle_{L^2(S)}|\rightarrow 0$ and $\|G((\phi_{j,n})_2)\|^2_{L^2(S)}\sim d_n^2$, $j=1,2$ as $n\rightarrow \infty$, since
\begin{equation}\label{phi}
    (\phi_{j,n})_2=\begin{cases}
d_n e^{i n x}, \quad \mbox{for } j=1,\\
d_n e^{-i n x}, \quad \mbox{for } j=2,
\end{cases}
\end{equation}
where $0<m<|d_n|<M$ for $m$, $M>0$ independent of $n$ (c.f. \eqref{sp2}, \eqref{sp3}, \eqref{v}, \eqref{sp5} and Theorem \ref{sa}). Hence, there exists a $\varepsilon>0$ such that, for $n=\pm 1, \pm2,...$,
\begin{equation}\label{d1}
|\Delta_n|>\varepsilon.
\end{equation}
Therefore, we can apply the Cramer's rule to \eqref{h11}-\eqref{h12} and it follows
\begin{equation}\label{c}
    c_{1,n}=\frac{\Delta_{n,1}}{\Delta_n}, \quad     c_{2,n}=\frac{\Delta_{n,2}}{\Delta_n},
\end{equation}
with
\[\Delta_{n,1}=\begin{vmatrix}
-\alpha_{1,n} & \langle G((\phi_{2,n})_2), G((\phi_{1,n})_2) \rangle_{L^2(S)} \\
-\alpha_{2,n}& \langle G((\phi_{2,n})_2), G((\phi_{2,n})_2) \rangle_{L^2(S)}
\end{vmatrix}\]
and
\[\Delta_{n,2}=\begin{vmatrix}
\langle G((\phi_{1,n})_2), G((\phi_{1,n})_2) \rangle_{L^2(S)} &-\alpha_{1,n} \\
\langle G((\phi_{1,n})_2), G((\phi_{2,n})_2) \rangle_{L^2(S)} &-\alpha_{2,n}.
\end{vmatrix}.\]
We now have the explicit formula of control $h$ that satisfies \eqref{G5}.

It then remains to show that the control function $h$ defined by \eqref{h} and \eqref{c} belongs to $L^2(0,T;H^s(S))$. To prove that, we first write the standard Fourier expansions
\begin{equation}\label{g}
    G((\phi_{1,n})_2)(x)=\sum_{k=-\infty}^\infty a_{nk} e^{ikx},\quad G((\phi_{2,n})_2)(x)=\sum_{k=-\infty}^\infty b_{nk} e^{ikx},
\end{equation}
where
\[a_{nk}=\int_S G((\phi_{1,n})_2) e^{-ikx}dx, \quad b_{nk}=\int_S G((\phi_{2,n})_2) e^{-ikx}dx,\]
for $n=\pm1,\pm2,...$. Substituting these into \eqref{h} leads to (note that $c_0=0$)
\begin{equation}\label{h2}
    h(x,t)=\sum_{n\neq 0} q_n\left(c_{1,n} \sum_{k=-\infty}^\infty a_{nk} e^{ikx}+c_{2,n} \sum_{k=-\infty}^\infty b_{nk} e^{ikx}\right),
\end{equation}
thus,
\begin{align}\label{h3}
    \|h\|^2_{L^2(0,T;H^s(S))}=\int^T_0 \sum^\infty_{k=-\infty} (1+|k|)^{2s}\Big|\sum_{n\neq 0} q_n(t)&(c_{1,n}a_{nk}\\
    &+c_{2,n}b_{nk})\Big|^2dt.\nonumber
\end{align}
It then suffices to show that the following is finite,
\begin{align}\label{h4}
    &\int^T_0 \sum^\infty_{k=-\infty} (1+|k|)^{2s}\left|\sum_{n\neq 0} q_nc_{1,n}a_{nk}\right|^2dt\nonumber\\
    = & \sum^\infty_{k=-\infty} (1+|k|)^{2s}\int^T_0 \left|\sum_{n\neq 0} q_nc_{1,n}a_{nk}\right|^2dt\nonumber\\
    \leq& C\sum^\infty_{k=-\infty} (1+|k|)^{2s}\sum_{n\neq 0} |c_{1,n}|^2|a_{nk}|^2\nonumber\\
    \leq& C\sum_{n\neq 0} |c_{1,n}|^2\sum^\infty_{k=-\infty} (1+|k|)^{2s}|a_{nk}|^2
\end{align}
where the constant $C$ comes from the Riesz basis property of $\mathcal L$ in $P_T$. Combining \eqref{phi}, we are then able to repeat the proof in Theorem 1.1 of \cite{73} (or Theorem 2.1 of \cite{133}) and show that
\begin{equation}\label{ank}
\sum^\infty_{k=-\infty} (1+|k|)^{2s}|a_{nk}|^2 \leq C ((1+|n|)^{2s}+|g_n|^2)\|g\|_s^2,
\end{equation}
where $g$ is a smooth function given in the definition of $Gh$ and we set
\[g=\sum^\infty_{n=-\infty}g_n e^{inx}.\]
 Proceeding with the inequality \eqref{h4}, it leads to
\begin{align}\label{e1}
&C\sum_{n\neq 0} |c_{1,n}|^2\sum^\infty_{k=-\infty} (1+|k|)^{2s}|a_{nk}|^2\nonumber\\
\leq & C\|g\|_s^2 \sum_{n\neq 0}|c_{1,n}|^2  ((1+|n|)^{2s}+|g_n|^2)\nonumber\\
\leq & C \varepsilon^{-2}\|g\|_s^2 \sum_{n\neq 0} ((1+|n|)^{2s}+|g_n|^2)(m^2_{1,n}|\alpha_{1,n}|^2+m^2_{2,n}|\alpha_{2,n}|^2)\nonumber\\
\leq & CK^2 \varepsilon^{-2}\|g\|_s^2\sum_{n\neq 0}(1+|n|)^{2s}(|\alpha_{1,n}|^2+|\alpha_{2,n}|^2)
\end{align}
where
\[m_{1,n}=\|G((\phi_{2,n})_2)\|^2_{0}, \quad m_{2,n}=|\langle G((\phi_{2,n})_2), G((\phi_{1,n})_2) \rangle_{L^2(S)}|,\]
since it can be  checked that there exists a $K>0$  such that $m_{1,n}$, $m_{2,n}<K$ for any $n$ through the definition of the operator $G$ and \eqref{phi}.

Recall that  $\alpha_{j,n}=\langle\vec{u}_0,  \phi_{j,n}\rangle_{Q}$ with $\vec{u}_0=\begin{bmatrix}
   \varphi_0 \\
    \psi_0\end{bmatrix}$
 and $\phi_{j,n}=\vec{\eta}_{j,n}/\|\vec{\eta}_{j,n}\|_{X^s}$ defined in Theorem \ref{sa}. In addition, from \eqref{v} one has
\[\vec{\eta}_{j,n}=\begin{bmatrix}
   1/n^3 \\
    \lambda_{j,n}/n^3\end{bmatrix}e^{inx},\]
    for $j=1,2$ and $n=\pm1, \pm2,...$. Therefore,
    \begin{align*}
       |\alpha_{j,n}|&\leq C\left|\langle \begin{bmatrix}
   \varphi_0 \\
    \psi_0\end{bmatrix}, \begin{bmatrix}
    1/n^3 \\
    \lambda_{j,n}/n^3\end{bmatrix}e^{inx}\rangle_{Q}\right| \\
    &\leq C (|\varphi^0_n|+|\psi^0_n|),
    \end{align*}
   where $\varphi_n$ and $\psi_n$ are the coefficients of  Fourier series expansions of $\frac{1}{n^3}\varphi_0$ and $\psi_0$, that is,
   \[\varphi_0(x)=\sum^\infty_{n=-\infty} \varphi^0_n \frac{1}{n^3}e^{inx}, \quad \psi_0(x)=\sum^\infty_{n=-\infty} \psi^0_n e^{inx},\]
   in which we shall notice that
   \[\varphi^0_n=\int_S \varphi_0(x) \frac{1}{n^3}e^{-inx}dx=-i\int_S \varphi_0^{(3)}(x) e^{-inx}dx.\]
   Continue with \eqref{e1}, one has
   \begin{align*}
    &\int^T_0 \sum^\infty_{k=-\infty} (1+|k|)^{2s}\left|\sum_{n\neq 0} q_nc_{1,n}a_{nk}\right|^2dt\\
    \leq & CK^2 \varepsilon^{-2}\|g\|_s^2\sum_{n\neq 0}(1+|n|)^{2s}(|\varphi^0_n|^2+|\psi^0_n|^2)\\
    \leq & CK^2 \varepsilon^{-2}\|g\|_s^2 (\|\varphi_0\|^2_{s+3}+\|\psi_0\|^2_s).
   \end{align*}
   Hence, the control $h$ defined in \eqref{h} belongs to $L^2(0,T; H^s(S))$.

   Now, we consider the general case, $[\varphi_0]=\kappa\neq 0$. We still treat the control problem as finding a control $h$ such that the terminal state $\vec{u}(T)=\vec{0}$. Then, by setting $\vec{u}=\vec{v}+\vec{w}$ with $\vec{v}(T)=\vec{w}(T)=\vec{0}$,  the original control problem can be separated as two to find $K_{1,T}$ and $K_{2,T}$ such that:
   \begin{equation}\label{con01}
\begin{cases}
    \frac{d}{dt}\vec{v}(t)=A\vec{v}(t)+B K_{1,T}(\vec{u}_{1,0},\vec{0}),\\
    \vec{v}(0)=\vec{u}_{1,0},
    \end{cases}
\end{equation}
and
\begin{equation}\label{con02}
\begin{cases}
    \frac{d}{dt}\vec{w}(t)=A\vec{w}(t)+B K_{2,T}(\vec{u}_{2,0},\vec{0}),\\
    \vec{w}(0)=\vec{u}_{2,0},
    \end{cases}
\end{equation}
where $\vec{u}_{1,0}=\vec{u}_0-(\kappa,0)^T$ and $\vec{u}_{2,0}=(\kappa,0)^T$, satisfy
\[\vec{v}(T)=\vec{0}, \quad \vec{w}(T)=\vec{0}.\]
For problem \eqref{con01}, since $[\vec{u}_{1,0}]=0$, $K_{1,T}$ can be obtained according to the previous proof. For problem \eqref{con02}, we can find solution $\vec{w}=(\kappa,0)^T$ by simply setting $K_{2,T}=\vec{0}$.
   The proof is now complete.
\end{proof}

\subsection{Nonlinear problem}
Finally, we move to consider the nonlinear problem,
\begin{equation}\label{nonl}
\begin{cases}
    &u_{tt}-u_{xx}+\beta u_{xxxx}-u_{xxxxxx}+ (u^2)_{xx}= Gh,\quad x\in S,\\
    &u(x,0)=\varphi_0(x), u_t(x,0)=\psi_0(x),
    \end{cases}
\end{equation}
and show that the system is exact controllable.

 \begin{proof}(\textbf{Theorem \ref{c2}})
 To begin with the proof, we write IVP as
 \begin{equation}\label{s0}
 \begin{cases}
    &\frac{d}{dt}\vec{u}=A\vec{u}+F(\vec{u})+B\vec{h},\\
     &\vec{u}(0)=\vec{u}_0,
    \end{cases}
 \end{equation}
 where $A$, $B$ and $F$ are defined in \eqref{w2}. One can write the solution of this first order system as
    \begin{equation}\label{s1}
        \vec{u}(t)=W(t)\vec{u}_0+\int^t_0 W(t-\tau)F(\vec{u})(\tau)d\tau+\int^t_0 W(t-\tau)( B\vec{h})(\tau)d\tau.
    \end{equation}
  We define
  \begin{equation}\label{s2}
    \mu(T,\vec{u})=\int^T_0 W(T-\tau)F(\vec{u})(\tau)d\tau,
  \end{equation}
For given $\vec{u}_0$, $\vec{u}_T \in {\mathcal X}_s$, we set
  \[h=K_T(\vec{u}_0, \vec{u}_T-\mu(T,\vec{u})), \]
  as it is defined in Theorem \ref{c0}. Therefore, according to Theorem \ref{c0}, one has
  \[\vec{u}_T-\mu(T,\vec{u})=W(T)\vec{u}_0+\int^T_0 W(T-\tau)(B\vec{h})(\tau) d\tau ,\]
  that is,
  \[\vec{u}_T=W(T)\vec{u}_0+\int^T_0 W(T-\tau)( B\vec{h})(\tau)d\tau+\int^T_0 W(T-\tau)F(\vec{u})(\tau)d\tau=\vec{u}(T).\]
  In addition, according to \eqref{s1}, we can always have $\vec{u}(0)=\vec{u}_0$. It then suffices to show that, for given $T>0$, the map
  \begin{align*}
  \Gamma(\vec{u})=W(t)\vec{u}_0+\int^t_0 &W(t-\tau)F(\vec{u})(\tau)d\tau\\
  &+\int^t_0 W(t-\tau)\Big( BK_T\big(\vec{u}_0, \vec{u}_T-\mu(T,\vec{u})\big)\Big)(\tau)d\tau
  \end{align*}
  is contraction in an appropriate space.

  We denote
   \[S_r:=\{\vec{v}\in X^s| \|\vec{v}\|_{X^s}\leq r\},\]
   and we aim to show the map $\Gamma$ is contraction from $S_r$ to $S_r$ for proper $r$ depending on $T>0$. According to Proposition \ref{p1}, one has
   \begin{align*}
    \|\Gamma(\vec{u})\|_{X^s}\leq& C\|\vec{u}_0\|_{X^s}+ C\int^{T}_0 (\|F(\vec{u})\|_{X^s}+\|BK_T(\vec{u}_0, \vec{u}_T-\mu(T,\vec{u}))\|_{X^s})d\tau\\
    \leq &C\|\vec{u}_0\|_{X^s}+ C\sup_{t\in[0,T]}\|\vec{u}\|^2_{X^s}+C \Big(\|\vec{u}_0\|_{X^s}+\|\vec{u}_T\|_{X^s}\\
    &+\|\mu(T,\vec{u})\|_{X^s}\Big).
   \end{align*}
   Moreover, one has
   \begin{align*}
   \|\mu(T,\vec{u})\|_{X^s}&\leq \left\|\int^T_0 W(T-\tau)\begin{bmatrix}
   0 \\
    -(u^2)_{xx}\end{bmatrix}d\tau\right\|_{X^s} \leq C \sup_{t\in[0,T]}\|\vec{u}\|^2_{X^s}.
   \end{align*}
   Therefore, we have
   \[\|\Gamma(\vec{u})\|_{X^s}\leq C(\|\vec{u}_0\|_{X^s}+\|\vec{u}_T\|_{X^s})+C\|\vec{u}\|^2_{X^s}. \]
   For $\|\vec{u}_0\|_{X^s}<\delta$ and $\|\vec{u}_T\|_{X^s}<\delta$, by choosing $\delta$ and $r$ such that,
\begin{equation}\label{cmap1}
    2C\delta+C r^2\leq r, \quad Cr\leq \frac12,
\end{equation}
one has
\[\|\Gamma(\vec{u})\|_{X^s}\leq r.\]
   In addition, we have
   \begin{align*}
   &\|\Gamma(\vec{u})-\Gamma(\vec{v})\|_{X^s}\\
   \leq &\Big\|\int^T_0 F(\vec{u})-F(\vec{v})d\tau \Big\|_{X^s}+\Big\|\int^T_0 BK_T(0, \mu(T,\vec{u})-\mu(T,\vec{v}))d\tau \Big\|_{X^s}\\
   \leq &C \Big(\sup_{t\in [0,T]}\|\vec{u}\|_{X^s}+\sup_{t\in [0,T]}\|\vec{v}\|_{X^s}\Big)\|\vec{u}-\vec{v}\|_{X^s}\\
   \leq &C r \|\vec{u}-\vec{v}\|_{X^s}.
   \end{align*}
  According to \eqref{cmap1}, one obtains,
   \[\|\Gamma(\vec{u})-\Gamma(\vec{v})\|_{X^s}\leq \frac12  \|\vec{u}-\vec{v}\|_{X^s}.\]
  Therefore we have the desired  contraction mapping conclusion and the proof is complete.
 \end{proof}
\section{stabilization problem}
In this section, we study the closed loop system
\begin{equation}\label{close1}
    \begin{cases}
    u_{tt}-u_{xx}+\beta u_{xxxx}-u_{xxxxxxx}+(u^2)_{xx}=-\K Gu_t,  \ \ \ x\in S, \\
    u(x,)=\varphi_0(x), u_t(x,0)=\psi_0(x),\\
    \end{cases}
\end{equation}
where $[\psi_0]=0$ and $\K>0$.
\subsection{Linear problem}
We start to show the exponential decay result for the linear system
\begin{equation}\label{closelin}
    \begin{cases}
    u_{tt}-u_{xx}+\beta u_{xxxx}-u_{xxxxxxx}=-\K Gu_t,  \ \ \ x\in S,\\
    u(x,0)=\varphi_0(x), u_t(x,0)=\psi_0(x),\\
    \end{cases}
\end{equation}
where $[\psi_0]=0$.
\begin{thm}\label{expd}
Given   $s\geq 0$, for any $(\varphi_0, \psi_0)\in X^s$ with $[\psi_0]=0$, the system (\ref{closelin}) admits a unique solution $u\in C(\R, H^s(S))$. Moreover, there exists $C, \gamma >0 $  such that
\begin{equation}\label{linest}
    \|(u(\cdot,t), u_t(\cdot,t))-([\varphi_0],0)\|_{X^s}\leq C e^{-\gamma t}\|(\varphi_0,\psi_0)-([\varphi_0] ,0)\|_{X^s}.
\end{equation}
\end{thm}

%
\begin{proof}  The existence of the solution follows from a standard semigroup theory (c.f.\cite{123}). To establish the estimate \eqref{linest}, with out loss of generality, we can assume $[\varphi_0]=0$   since $u(x,t)-[\varphi_0]$ is also a solution of \eqref{closelin}. We start to show that the estimate holds for $s=0$.
Set
\begin{equation}\label{E0}
    E(t)=\frac12 \int_S u^2_{t}+u_x^2+\beta u_{xx}^2+u_{xxx}^2dx,
\end{equation}
then it suffices to show
\begin{equation}\label{E}
    E(t)\leq C e^{-\gamma t}E(0).
\end{equation}
Given $T>0$, using integration by parts yields,
\begin{align}
 E(T)-E(0)=&\frac{1}{2}\int^T_0\frac{d}{dt}\int_S     u^2_{t}(\cdot,t)+u^2_{x}(\cdot,t)+\beta u^2_{xx}(\cdot,t) + u^2_{xxx}(\cdot,t)dxdt\nonumber\\
 =& \int^T_0\int_S u_t u_{tt}+u_xu_{xt}+\beta u_{xx}u_{xxt}+u_{xxx}u_{xxxt}dxdt\nonumber\\
 =& \int^T_0\int_S u_t (u_{xx}-\beta u_{xxxx}+u_{xxxxxx}-\K Gu_t)+u_xu_{xt}\nonumber\\
 &+\beta u_{xx}u_{xxt}+u_{xxx}u_{xxxt}dxdt\nonumber\\
 =& -\K \int^T_0 \int_S u_t Gu_t dxdt\nonumber \\
 =& -\K \int^T_0\int_S u_t g\left(u_t-\int_S g(s) u_t(\cdot,t) ds\right)dxdt\nonumber\\
 =& -\K \int^T_0 \int_S g\left(u_t-\int_S g(s)u_t(\cdot,t) ds\right)^2dxdt,\label{et0}
\end{align}
the last equality holds since $[g]=1$ and
\begin{align}
&\int_S g \left(u_t \int_S g(s)u_t(\cdot,t)ds-\left(\int_S g(s)u_t(\cdot,t)ds\right)^2 \right)dx\\
=& \left(1-\int_S g(s) ds\right)\left(\int_S g(s)u_t(\cdot,t)ds\right)=0\nonumber.
\end{align}
Next, since $((0,0),(u_t(x,0),u_t(x,T)))\in {\mathcal X}_0$, according to Theorem \ref{c0},  we can have control $G f\in L^2(0,T; L^2(S))$ such that the system
\begin{equation}\label{v0}
    \begin{cases}
    v_{tt}-v_{xx}+\beta v_{xxxx}-v_{xxxxxx}=Gf,\ \ \ x\in S, \\
    v(x,0)=0, v_{t}(x,0)=0,
    \end{cases}
\end{equation}
admits a solution satisfying
\[
    v(x,T)=u(x,T), v_{t}(x,T)=u_t(x,T).
 \]
 with
\begin{equation}\label{v201}
    \|Gf\|_{L^2(0,T;L^2(S))}\leq C\|(u(x,T),u_t(x,T))\|_{X^0}.
\end{equation}
Due to the boundedness of the operator $G$, for each $t$, one has
 \begin{equation}\label{v20}
    \|f(\cdot,t)\|_{L^2(0,T;L^2(S))}\leq  \|Gf(\cdot,t)\|_{L^2(0,T;L^2(S))}\leq  C\|(u(x,T),u_t(x,T))\|_{X^0}.
\end{equation}
In addition, since
\[\|Gf\|_{L^1(0,t;L^2(S))}\leq t^\frac12\|Gf\|_{L^2(0,t;L^2(S))},\]
according to Theorem \ref{f02}, for $t\in[0,T]$, we obtain,
\begin{equation}\label{v10}
    \|v(\cdot,t)\|_{L^2(0,T;L^2(S))}\leq C\|(u(x,T),u_t(x,T))\|_{X^s}.
\end{equation}
 We now consider
\begin{align*}
E(T)=&\frac12\int_S u_t^2(x,T)+u_x^2(x,T)+\beta u_{xx}^2 (x,T)+ u_{xxx}^2(x,T)dx\\
=& \frac12 \int_S u_t(x,T)v_t(x,T)+ u_{xx}(x,T)v_{xx}(x,T)+\beta u_{xxx}(x,T)v_{xxx}(x,T)\\
&+u_{xxx}(x,T)v_{xxx}(x,T)dx -\frac12 \int_S u_t(x,0)v_t(x,0)+ u_{xx}(x,0)v_{xx}(x,0)\\
&+\beta u_{xxx}(x,0)v_{xxx}(x,0)+u_{xxx}(x,0)v_{xxx}(x,0)dx\\
=&\int_0^T \frac{d}{dt}\int_S u_{t}v_{t}+u_{x}v_{x}+\beta u_{xx}v_{xx}+u_{xxx}v_{xxx}dxdt\\
=& \int^T_0 \int_S v_t(u_{xx}-\beta u_{xxxx}+u_{xxxxxx}-\K Gu_t)+u_t(v_{xx}-\beta v_{xxxx}\\
&+v_{xxxxxx}+Gf) + (u_{x}v_{x})_t+\beta (u_{xx}v_{xx})_t+(u_{xxx}v_{xxx})_tdxdt\\
=& \int^T_0\int_S -\K v_t Gu_t+u_tGfdxdt=\int^T_0 \int_S G u_t(f-\K v_t)dxdt\\
\leq &\|G u_t\|_{L^2(0,T;L^2(S))}\|f-\K v_t\|_{L^2(0,T;L^2(S))},
\end{align*}
since $G$ is self-adjoint in $L^2(S)$. According to the setup of $g$ in the operator $G$, we denote $g^* >0$ as the least upper bound of $g(x)$ in $S$. In addition, combining \eqref{v20} and \eqref{v10}, we have
\[\|f-\K v_t\|_{L^2(0,T;L^2(S))}\leq C\|(u(\cdot,T),u_t(\cdot,T))\|_{X^s}\leq C (E(T))^\frac12.\] Thus,
\begin{align}
    E(T) \leq &  \left[Cg^* \int^T_0 \int_S g(x)\left( u_t-\int_S g(s) u_t(d,t)ds\right)^2dxdt\right]^\frac12(E(T))^\frac12.
\end{align}
It then follows,
\begin{equation}\label{et}
    \int^T_0 \int_S g(x)\left( u_t-\int_S g(s) u_t(d,t)ds\right)^2dxdt\geq (C g^* )^{-1}E(T).
\end{equation}
Substitute \eqref{et} into \eqref{et0}, we obtain,
\begin{equation}\label{et1}
    E(T)-E(0)\leq -\K (C g^*)^{-1} E(T),
\end{equation}
hence,
\begin{equation}\label{et2}
    E(T)\leq \frac{C g^*}{\K+ C g^*} E(0):=r E(0),
\end{equation}
where $0<r<1$. Repeating this estimate on successive intervals $[(k-1)T,kT]$, for $k=2,3,...$, with
$$(u(x,0), u_t(x,0)),\quad (u(x,T), u_t(x,T))$$
replaced by
$$(u(x,(k-1)T), u_t(x,(k-1)T)),\quad (u(x,kT), u_t(x,kT)),$$
it follows,
\begin{equation}\label{et3}
    E(kT)\leq r^k E(0)=e^{kT (\frac1T\ln r)}E(0).
\end{equation}
Therefore, we complete the proof for $s=0$.

Next, we move on to the case $s=6$ by considering $(\varphi_0,\psi_0)\in X^6$. Set $v=u_t$ and $w=v_t$, then $v$ solves the system
\begin{equation}\label{v0}
\begin{cases}
    v_{tt}-v_{xx}+\beta v_{xxxx}-v_{xxxxxx}=-\K Gv_t,\ \ \ x\in S,\\
    v(x,0)=\psi_0(x), v_t(x,0)=\varphi_1(x),
    \end{cases}
\end{equation}
and $w$ solves the system
\begin{equation}\label{w0}
\begin{cases}
    w_{tt}-w_{xx}+\beta w_{xxxx}-w_{xxxxxx}=-\K Gw_t,\ \ \ x\in S,\\
    w(x,0)=\varphi_1(x), w_t(x,0)=\psi_1(x),
    \end{cases}
\end{equation}
where
\[\varphi_1= \varphi_0^{(2)}-\beta \varphi_0^{(4)}+\varphi_0^{(6)}-\K G \psi_0\]
and
\[\psi_1=\psi_0^{(2)}-\beta \psi_0^{(4)}+\psi_0^{(6)}-\K G \varphi_1.\]
Thus, according to \eqref{linest} for $s=0$ which we have just established (it can be checked that $[\varphi_1]=[\psi_1]=0$),  one has
\begin{equation}\label{v1}
    \|(v,v_t)\|_{X^0}\leq C e^{-\gamma t}\|(\psi_0, \varphi_1)\|_{X^0}
\end{equation}
and
\begin{equation}\label{w1}
    \|(w,w_t)\|_{X^0}\leq C e^{-\gamma t}\|(\varphi_1, \psi_1)\|_{X^0}.
\end{equation}
Since $u_{xxxxxx}-\beta u_{xxxx}+u_{xx}=w+\K G v$, we conclude that
\begin{equation}\label{utt}
    \|u\|_{H^6(S)}\leq M e^{-\gamma t}\|(\varphi_0, \psi_0)\|_{X^6}.
\end{equation}
In addition, according to \eqref{v1}, we have
\begin{equation}\label{uttt}
    \|u_t\|_{H^3(S)}=\|v\|_{H^3(S)}\leq M e^{-\gamma t}\|(\varphi_0, \psi_0)\|_{X^6}.
\end{equation}
Therefore, estimate \eqref{linest} holds for $s=6$. We can show that \eqref{linest} holds for $s\in 6\N$ through induction, then the estimate for other $s$ follows from a classical interpolation argument.
\end{proof}

We notice that if we rewrite the system \eqref{closelin} as a first order evolution system
\begin{equation}\label{cll01}
    \begin{cases}
    \frac{d}{dt}\vec{u}=(A - \K B) \vec{u}, \quad x\in S,\\
    \vec{u}(x,0)=\vec{u}_0(x),
    \end{cases}
\end{equation}
where $v=u_t$, $\vec{u}=(u,v)^T$, $\vec{u}_0:=(\varphi_0,\psi_0)^T$ and $A$, $B$ are defined in (\ref{G2}). Its solution $\vec{u}$ can be written as
\begin{equation}\label{kb}
    \vec{u}(x,t)=W_{\K}(t)\vec{u}_0
\end{equation}
where $W_\K(t)$ is the strongly continuous semigroup on $X^s$ associated to \eqref{cll01}. The following corollary follows directly from Theorem \ref{expd}.
\begin{cor}\label{kb0}
Given   $s\geq 0$, for any $\vec{u}\in X^s$ with $[\psi_0]=0$, the system (\ref{cll01}) admits a unique solution $\vec{u}\in C(\R, X^s)$. Moreover, there exists $C, \gamma >0 $  such that
\begin{equation}\label{link}
\|W_\K(t)\vec{u}_0-([\varphi_0],0)\|_{X^s}\leq C e^{-\gamma t}\|\vec{u}_0-([\varphi_0],0)\|_{X^s}.
\end{equation}
\end{cor}

 \subsection{Stability  of the nonlinear system}

We now move on to  consider  the nonlinear closed loop system
\begin{equation}\label{cl0}
    \begin{cases}
    u_{tt}-u_{xx}+\beta u_{xxxx}-u_{xxxxxxx}+(u^2)_{xx}=-\K Gu_t,  \ \ \ x\in S,\\
    u(x,)=\varphi_0(x), u_t(x,0)=\psi_0(x),\\
    \end{cases}
\end{equation}
 or
\begin{equation}\label{cl1}
    \begin{cases}
    \frac{d}{dt}\vec{u}=(A - \K B) \vec{u} -F(\vec{u}), \quad x\in S,\\
    \vec{u}(x,0)=\vec{u}_0(x),
    \end{cases}
\end{equation}
 where  $A$, $B$ and $F$ are defined in (\ref{G2}).
 Before we prove Theorem \ref{stable1}, we state a lemma that will be cited later.

\begin{lem}\label{cll1}
For $s\geq 0$, the group of bounded operators $W_\K(t)$ introduced in (\ref{kb}) satisfies the following the relation to $W(t)$ (introduced in Proposition \ref{p1}),
\begin{equation}\label{wk1}
    W_\K(t)\vec{u}_0=W(t)\vec{u}_0-\K \int^t_0 W(t-\tau)B(W_\K(\tau)\vec{u}_0)d\tau
\end{equation}
 for any $\vec{u}_0\in X^s$. In addition, for $\vec{f}\in L^1(\R, X^s)$,
\begin{align}
    \int^t_0 W_\K (t-\tau) \vec{f} d\tau =\int^t_0& W(t)(t-\tau)\vec{f}d\tau\nonumber\\
    &-\K\int^t_0 W(t-\tau)B\left(\int^\tau_0 W_\K (\tau-\xi)\vec{f}d\xi\right)d\tau \label{wk2}.
\end{align}
\end{lem}
The proof of the lemma simply  based on the ``variation of parameters'' formula for the definition of the operators $W(t)$ and $W_\K(t)$, therefore omitted.

According to Lemma \ref{cll1}, we can rewrite the closed loop system \eqref{cl1} as an integral equation
\begin{align}
    \vec{u}(t)=& W_\K(t) \vec{u}_0-\int^t_0 W_\K(t-\tau)\left(F(\vec{u})\right)d\tau \nonumber\\
    =& W_\K(t)\vec{u}_0 -\int^t_0 W(t-\tau)(F(\vec{u}))(\tau)d\tau\label{cl2}\\
     &+\K\int^t_0 W(t-\tau)B\left(\int^\tau_0 W_\K (\tau-\xi)F(\vec{u})(\xi)d\xi\right)d\tau \nonumber.
\end{align}

\begin{proof}(\textbf{Theorem \ref{stable1}})
For simplicity, we consider the problem by assuming $[\varphi_0]=0$. According to Corollary \ref{kb0}, for any $\vec{u}_0\in X^s$, by choosing $T>0$ such that
\begin{equation}\label{link1}
C e^{-\gamma T}\leq \frac14,
\end{equation}
one can have
\begin{equation}\label{link}
\|W_\K(t)\vec{u}_0\|_{X^s}\leq \frac14\|\vec{u}_0\|_{X^s},
\end{equation}

We then seek a solution $u$  to the integral equation
\begin{align}
    \Gamma(\vec{u})=&W_\K(t)\vec{u}_0 -\int^t_0 W(t-\tau)(F(\vec{u}))(\tau)d\tau\label{cl3}\\
     &+\K\int^t_0 W(t-\tau)B\left(\int^\tau_0 W_\K (\tau-\xi)F(\vec{u})(\xi)d\xi\right)d\tau \nonumber
\end{align}
in some ball $S_r$  for any $t\in [0,T]$ where
  \[S_r:=\{\vec{v}\in X^s| \|\vec{v}\|_{X^s}\leq r\}.\]
This will be done provided that $\|\vec{u}_0\|_{X^s}\leq \delta $ where $\delta$ is determined later. Furthermore,  to ensure the exponential stability, $\delta$ and $r$ will be chosen such that $\|\vec{u}(T)\|_{X^s}\leq \frac12 \|\vec{u}_0\|_{X^s}$.

From the properties of $W_\K(t)$ as a strong continuous semigroup on $X^s$, $s\geq 0$, we can obtain for given $T>0$,
    \begin{equation}\label{wke1}
        \sup_{t\in[0,T]}\|W_\K(t)\vec{u}\|_{X^s}\leq C \|\vec{u}\|_{X^s},
    \end{equation}
    \begin{equation}\label{wkw2}
        \sup_{t\in[0,T]} \left\|\int^t_0 W_\K(t-\tau)\vec{f}d\tau\right\|_{X^s}\leq C\|\vec{f}\|_{L^1(0,T;X^s)}.
    \end{equation}
Similar estimates holds for $W(t)$ (c.f. Proposition \ref{p1}).
Recall that $F(\vec{u})=(0,-(u^2)_{xx})^T$, we then have, for $0<t<T$,
\begin{align}
   \left\| \int^t_0 W (\tau-\xi)F(\vec{u})(\xi)d\xi\right\|_{X^s}\leq & C\|F(\vec{u})\|_{L^1(0,T;X^s)}\label{wke3}\\
   \leq & C \sup_{t\in[0,T]} \|(u^2)_{xx}\|_{s}\nonumber\\
   \leq &C  \sup_{t\in[0,T]} \|u\|^2_{s+3}\nonumber\\
   \leq & C  \sup_{t\in[0,T]}\|\vec{u}\|^2_{X^{s}}\nonumber
\end{align}
In addition, due to the boundedness of $G$, one has
\begin{equation*}
   \left\|B\left(\int^\tau_0 W_\K (\tau-\xi)F(\vec{u})(\xi)d\xi\right)\right\|_{X^s}\leq C \left\|\int^\tau_0 W_\K (\tau-\xi)F(\vec{u})(\xi)d\xi\right\|_{X^s}.
\end{equation*}
Then, these lead to
\begin{align}\label{wke5}
    &\left\|\int^t_0 W(t-\tau)B\left(\int^\tau_0 W_\K (\tau-\xi)F(\vec{u})(\xi)d\xi\right)d\tau\right\|_{X^s}\\
    \leq &C\left\|B\left(\int^t_0 W_\K (t-\xi)F(\vec{u})(\xi)d\xi\right)\right\|_{L^1(0,T;X^s)}\nonumber\\
    \leq & C  \sup_{t\in[0,T]}\|\vec{u}\|^2_{X^{s}}.\nonumber
\end{align}
Combining \eqref{wke1}-\eqref{wke5}, we have
\begin{align}
    \|\Gamma(\vec{u})\|_{X^s}\leq& C\|\vec{u}_0\|_{X^s}+C \sup_{t\in[0,T]}\|\vec{u}\|^2_{X^{s}}\nonumber\\
    \leq &C\|\vec{u}_0\|_{X^s}+ C r^2\label{contract1}
\end{align}
and
\begin{align}
    \|\Gamma(\vec{u})-\Gamma(\vec{v})\|_{ X^s}\leq& C\Big( \sup_{t\in[0,T]}\|\vec{u}\|_{X^s}+\sup_{t\in[0,T]}\|\vec{v}\|_{X^s}\Big) \|\vec{u}-\vec{v}\|_{X^s}\nonumber\\
    \leq & Cr \|\vec{u}-\vec{v}\|_{X^s}\label{contract2}.
\end{align}
for some $C>0$ independent $t$, $\delta$ and $r$.

On the other hand, based on \eqref{link} and \eqref{cl3}, one has
\begin{align}
    \|\Gamma(\vec{u})(T)\|_{X^s}\leq &\|W_\K(T)\vec{u}_0\|_{X^s}+\left\|\int^T_0 W(t-\tau)(F(\vec{u}))(\tau)d\tau\right\|_{X^s}\nonumber\\
     &+\K\left\|\int^T_0 W(t-\tau)B\left(\int^\tau_0 W_\K (\tau-\xi)F(\vec{u})(\xi)d\xi\right)d\tau\right\|_{X^S} \nonumber\\
     \leq &\frac14 \|\vec{u}_0\|_{X^s}+Cr^2\label{link2}.
\end{align}
We set $\delta=4Cr^2$ where $r>0$ is chosen so that
\[(4C^2+C)r\leq 1 \quad \mbox{and}\quad Cr\leq \frac12.\]
Therefore, we have $\Gamma$ is a contraction mapping in $S_r$. Moreover, according to \eqref{link2}, the unique fixed point $\vec{u}\in S_r$ satisfies
\[\|\vec{u}(T)\|_{X^s}=\|\Gamma(\vec{u})(T)\|_{X^s}\leq \frac{\delta}{2}.\]
Now, for $0<\|\vec{u}_0\|_{X^s}\leq \delta$, we  set $\delta'=\|\vec{u}_0\|_{X^s}$, By changing $\delta'$ into $\delta$ and $r$ into $r'=(\delta'/\delta)^{\frac12} r$, we can obtain
\[\|\vec{u}(T)\|_{X^s}\leq \frac{\delta'}{2}=\frac12 \|\vec{u}_0\|_{X^s}.\]
Hence, it follows through induction,
\[\|\vec{u}(nT)\||_{X^s}\leq 2^{-n} \|\vec{u}_0\|_{X^s}.\]
Similar to \eqref{et3} in Theorem \ref{expd}, we then have
\[\|\vec{u}\|_{X^s}\leq M e^{-\sigma t}\|\vec{u}_0\|_{X^s}\]
for some constant $M>0$ and $\sigma>0$.  The proof is now complete.
\end{proof}

\begin{rem}
For the proof of $[\varphi_0]\neq 0$, as we shown in the introduction, the quantity $[u(\cdot,t)]$ is conserved for any $t$. We then can set $[\varphi_0]=\eta$ and adapt the idea from Zhang's work (Theorem 1.2 in \cite{73}) by  re-considering $v(x,t)= u(x,t)-\eta$ (which follows $[v(0,t)]=0$) that solves the new system
\begin{equation*}
    \begin{cases}
    v_{tt}-(1+2\eta)v_{xx}+\beta v_{xxxx}-v_{xxxxxx}+(v^2)_{xx}=-\K Gv_t,  \ \ \ x\in S,\\
    v(x,0)=\varphi_0(x)-\eta , v_t(x,0)=\psi_0(x).\\
    \end{cases}
\end{equation*}
The proof for the new system is similar to the one for $[\varphi_0]=0$ therefore omitted.
\end{rem}

\end{document}